\documentclass[10pt]{article}
\usepackage{graphicx}
\usepackage[utf8]{inputenc}
\usepackage{amsmath,color,sectsty,amsthm,amssymb,fontenc,color,tensor,textcomp,amsfonts,relsize,graphicx,graphics,enumerate}
\usepackage{fancyhdr}

\usepackage[colorlinks=true,
            linkcolor=red,
            urlcolor=blue,
            citecolor=green]{hyperref} 

\usepackage{mathtools}
\usepackage{amsmath,tikz}
\usepackage{tensor}
\usepackage[rmargin=1.9cm, lmargin=1.9cm]{geometry}
\usepackage{authblk}
\newtheorem{theo}{Theorem}[section]
\newtheorem{lem}[theo]{Lemma}
\newtheorem{exa}[theo]{Example}
\newtheorem{prop}[theo]{Proposition}
\newtheorem{cor}[theo]{Corollary}
\newtheorem{df}[theo]{Definition}
\newtheorem{rmk}[theo]{Remark}

\makeatletter
\newcommand{\address}[1]{\gdef\@address{#1}}
\newcommand{\email}[1]{\gdef\@email{\url{#1}}}
\newcommand{\@endstuff}{\par\vspace{\baselineskip}\noindent\small
\begin{tabular}{@{}l}\scshape\@address\\{E-mail address:} \@email \end{tabular}}
\AtEndDocument{\@endstuff}
\makeatother
\title{Generalised Local Fractional Hermite-Hadamard Type Inequalities on Fractal Sets }
\author{
Peter Olamide Olanipekun }
\address{Department of Mathematics, The University of Auckland, Auckland 1010, New Zealand.\\ E-mail address: \color{blue} olanipekunp@gmail.com}

\email{peter.olanipekun@auckland.ac.nz}

\date{}
\begin{document}\sloppy
\maketitle

\abstract{Fractal geometry and analysis constitute a growing field, with numerous applications,  based on the principles of fractional calculus. Fractals sets are highly effective in improving convex inequalities and their generalisations. In this paper, we establish a generalized notion of convexity. By defining generalised $\phi_{h-s}$ convex functions, we extend the well known concepts of generalised convex functions, $P$-functions, Breckner $s$-convex  functions, $h$-convex functions amongst others. With this definition, we prove Hermite-Hadamard type inequalities for generalized $\phi_{h-s}$ convex mappings onto fractal sets. Our results are then applied to probability theory. }

\section{Introduction}

Recent advancements have demonstrated the applications of fractional calculus in addressing real-life problems and enhancing a better understanding of intricate situations across various scientific fields. These applications span  fields, including, but not limited to, probability models \cite{oraby, lau}, physical modeling and experiments \cite{pot, boh, ciu}, control and dynamical systems \cite{liy, du, feliu}, image processing \cite{tian, pu}, robotics and signal processing \cite{gonz, dua}. Fractional calculus has contributed significantly to medical and biomedical research especially in neuroscience \cite{ris, kara, ingo, shi, pinto1, pinto2, val}. Recent investigations extend to economic risk analysis \cite{ro}, wind turbulence \cite{haj}, and ongoing discussions across various disciplines.  
It has been used to improve the accuracy of several models in science and technology, finance and economics, medicine and engineering etc. Models of this nature have proven to exhibit greater efficiency compared to integer-order models. Most research in this direction have considered generalising existing theories and results to fractional order via fractional calculus. For instance, fractional Fourier transform \cite{alie}, numerical and analytical fractional partial differential equations \cite{owo, rod}, $L^p$ theory and fractional order estimates \cite{muc},  operator theory \cite{kra}, fractional calculus of variations \cite{alme, del}, generalised solutions to fractional div-curl systems \cite{delg}  and more recently, their application to address challenges related to COVID-19 \cite{pand}. For more details on fractional calculus see the books \cite{old, miller}.

On the other hand, fractal geometry and analysis constitute a burgeoning area of study grounded in the principles of fractional calculus. Although the term ``fractal" was introduced by Beno\^it Mandelbrot in 1975 in his memoir,  later revised in 1982 \cite{man}, there is an extensive historical underpinning to fractals. Its rich history goes back to the time when mathematicians thought about certain curves and surfaces that differ conceptually from the classical ones in  geometry. Fractals occur in nature in various forms and scales, exhibiting self-similar patterns at different magnifications. The geometry and analysis of fractals have been well fairly examined in literature. In a study on fractal dimension, \cite{tri} established some remarkable properties for the complement of a fractal set by relating its Besicovitch-Taylor index to an exterior dimension with specific applications in porous materials, blood network and the boundary of a diffusion process. Although fractals appears in several areas in mathematics and statistics such as geometric measure theory, topology, harmonic analysis, differential equations, numerical analysis, time series analysis \cite{frank, por, kra, miller, owo, part}, fractal mathematics is still yet to be fully explored.

Inequalities are basic building blocks in mathematics and other related fields. Several works focused on integral inequalities aim to extend, refine, improve, and generalize known inequalities, while also introducing new ones. Fractional integral inequalities are useful tools in proving well-posedness and regularity of solutions to fractional partial differential equations \cite{wang}. One of the most celebrated inequalities for convex functions is the Hermite-Hadamard inequality named after Charles Hermite and Jacques Hadamard. Let $f:I \subset\mathbb{R} \rightarrow\mathbb{R}$ be a convex function defined on an interval $I$, then for $a,b\in I$, 
$$f\left(\frac{a+b}{2}\right)\leq \frac{1}{b-a} \int_a^b f(x)\, dx\leq \frac{f(a)+f(b)}{2}.$$

\section{Generalised Form of Convexity on Fractal Set}
In \cite{youn}, the author modified the concept of convex sets by introducing and developing the notion of $E$-convex sets, which we will refer to as the $\phi$-convex sets in order to align with our notation.  The author also defined $\phi$-convex functions \cite{youn} and established some interesting results which have been applied to convex programming problems  \cite{ant, fajardo, piao, syau} with notable extensions on Riemannian manifolds \cite{iqbal, kili}. Since then, several definitions extending the concept of convexity within various contexts have been provided in literature. 
 We recall the following definition from \cite{peterf}.
\begin{df}\label{fdg}
Let $I$ be an interval of $\mathbb{R}$ and $[a,b]\subseteq I$. A function $f:I\rightarrow \mathbb{R}^\alpha$ is said to be  $\phi_{h-s}$ convex if for all $x,y\in [a,b]$ and $s,t\in[0,1]$
\begin{align}
f(t\phi(x)+ (1-t)\phi(y)) \leq \left ( \frac{t}{h(t)} \right)^{ s} f(\phi(x)) + \left ( \frac{1-t}{h(1-t)} \right)^{s} f(\phi(y)).  \nonumber
\end{align}
\end{df}

In general, $\phi_{h-s}$ convex functions can be defined on $\phi$-convex set in $\mathbb{R}^n$ \cite{youn}. We remark that with $s=1$ and $h(t)=1$ in Definition \ref{fdg}, for all $s,t\in[0,1]$, we recover Definition 3.1 in \cite{youn} and the results therein.  Some properties of $\phi_{h-s}$ convexity were established in \cite{peterf} and several inequalities that generalise those of Jensen and Schur  were derived for some non-negative supermultiplicative functions. Later, the class of harmonically $\phi_{h-s}$ convex functions was given in \cite{peterfacta} and applied to the Hermite-Hadamard inequality with several analytical implications for special means of real numbers. Also, in \cite{peterfacta}, the authors proved some Hermite-Hadamard-type inequalities by applying the fractional integral operator to $\phi_{h-s}$ convex functions. The relationship between $\phi_{h-s}$ convex functions and their harmonic counterparts was explored in \cite{ojo}, where several Pachpatte-type inequalities were established under a ratio-bound condition for the harmonically $\phi_{h-s}$ convex functions.

Variations of Definition \ref{fdg} have been  used to establish Hermite-Hadamard-type inequalities on time scales \cite{fagbemigun1, fagbemigun2}. We adapt the above definition to fractal sets as follows.

\begin{df} \label{definition}
Let $I$ be an interval of $\mathbb{R}$ and let $[a,b]\subseteq I$. Let $\phi:[a,b]\rightarrow \mathbb{R}$ and $h:[0,1]\rightarrow (0,\infty).$  
A function $f:I\rightarrow \mathbb{R}^\alpha$ is said to be generalised $\phi_{h-s}$ convex if for all $x,y\in [a,b]$ and $s,t\in[0,1]$
\begin{align}
f(t\phi(x)+ (1-t)\phi(y)) \leq \left ( \frac{t}{h(t)} \right)^{\alpha s} f(\phi(x)) + \left ( \frac{1-t}{h(1-t)} \right)^{\alpha s} f(\phi(y)).  \label{deff}
\end{align}
\end{df}

If the inequality \eqref{deff} is reversed, then $f$ is said to be generalised $\phi_{h-s}$ concave, and of course $f$ is generalised $\phi_{h-s}$ convex if and only if $-f$ is generalised $\phi_{h-s}$ concave. If the inequality in \eqref{deff} is strict whenever $x$ and $y$ are distinct and $s,t\in[0,1]$, then $f$ is a {\it strictly generalised} $\phi_{h-s}$ convex.
\begin{rmk} \label{r1}
It is important to mention that the generalised class of $\phi_{h-s}$ convex function given in Definition \ref{definition} unifies several interesting generalised versions of convex functions \cite{dragomir, pengxu, mo, sun, tomar} some of which are highlighted below.
\begin{enumerate}[(i)]
\item If $s=0$ and $\phi(x)=x$, then $f$ is a generalised $P$-function \cite{dragomir}.
\item If $s=1$, $h(t)=1$ for all $t\in [0,1]$, and $\phi(x)=x$, then $f$ is a generalised convex function \cite{tomar}
\item If $h(t)=1$ for all $t\in [0,1]$,  $\phi(x)=x$ and $t^{s}+(1-t)^s=1$ with $s\in(0,1)$,  then $f$ is generalised Breckner $s$-convex in the first sense \cite{mo}. 
\item Let $s\in(0,1)$, if $h(t)=1$ for all $t\in [0,1]$,        and $\phi(x)=x$, then $f$ is generalised Breckner $s$-convex in the second sense \cite{mo}.
\item Define $\tilde{h}(t):= \frac{t}{h(t)}$, if $s=1$ and $\phi(x)=x$ then $f$ is a generalised $\tilde h$-convex function \cite{pengxu}.
\item If $h(t)=t^2$ for all $t\in [0,1]$ and $\phi(x)=x$ then $f$ is generalised $s$-Godunova-Levin function.
\item If $h(t)=2\sqrt{t-t^2}$ for all $t\in [0,1]$, then $f$ is generalised $s$-MT convex. Additionally, if $s=1$, then $f$ is generalised MT-convex.

\end{enumerate}

\end{rmk}
All the classes of convex functions listed above are \textit{generalised} in the sense that they map intervals of $\mathbb{R}$ into fractal set. Throughout this work, we will use the function $\rho_{ s}:[0,\infty)\rightarrow[0,\infty)$ defined by $\rho_s(t):=\left(\frac{t}{h(t)}\right)^{ s}$ with $h\not\equiv 0$, and the notation $\rho_{\alpha s}(t):= (\rho_s(t))^\alpha=\left(\frac{t}{h(t)}\right)^{\alpha s}$ . In most cases, we will restrict $\rho_{s}$ on the interval $[0,1]$. There are several interesting properties of this new notion of generalised convex function some of which we shall now mention. To this end, the following remarks are in place.

\begin{rmk} 
\begin{enumerate}[(i)]
\item  Let $h:[0,1]\rightarrow  (0,\infty)$ satisfies $\rho_s(t)\geq t$ for all $s,t\in[0,1]$. If $f:I\rightarrow\mathbb{R}^\alpha$ is generalised convex on $I$, then 
$$f(t\phi(x)+ (1-t)\phi(y)) \leq t^{\alpha} f(\phi(x))+ (1-t)^\alpha f(\phi(y))\leq \rho_{\alpha s}(t) f(\phi(x))+ \rho_{\alpha s}(1-t)f(\phi(y))$$
which implies that $f$ is also generalised $\phi_{h-s}$ convex. An example is the function $f:I\rightarrow \mathbb{R}$ defined by $f(x) =x^{\alpha }$. It is not difficult to see that
\begin{align}
f(t\phi(x)+ (1-t)\phi(y))&= t^{\alpha }(\phi(x))^{\alpha } +(1-t)^{\alpha }(\phi(y))^{\alpha }  \nonumber
\\& = t^\alpha f(\phi(x)) +(1-t)^\alpha f(\phi(y))  \nonumber
\\&\leq \rho_{\alpha s}(t) f(\phi(x)) + \rho_{\alpha s}(1-t) f(\phi(y)).  \nonumber
\end{align}
The Mittag-Leffler function \cite{pengxu, yang}  defined on $\mathbb{R}^\alpha$ by $E_\alpha(x^\alpha)=\sum_{k=0}^\infty \frac{x^{\alpha k}}{\Gamma(1+k\alpha)}$ for all $x\in\mathbb{R}$ is generalised convex \cite{tomar}, and hence by the remark above, a generalised $\phi_{h-s}$ convex function provided $\rho_s(t)\geq t$ for all $s,t\in[0,1]$.

\item Consider the functions $h_1, h_2:[0,1]\rightarrow (0,\infty)$ 
with $h_2(t)\leq h_1(t)$ for all $t\in [0,1]$. If $f$ is generalised $\phi_{h_1-s}$ convex then $f$ is generalised  $\phi_{h_2-s}$ convex. The converse is not true.

\item Let $f,g: I\rightarrow\mathbb{R}^\alpha$ be two generalised $\phi_{h-s}$ convex function and let $\lambda^\alpha\in\mathbb{R}^\alpha$ with $\lambda >0$ in $\mathbb{R}$.  Then the functions $f+g$ and $\lambda^\alpha f$ are both generalised $\phi_{h-s}$ convex functions on $I$.

\item Let $f: I\rightarrow\mathbb{R}^\alpha$ be a generalised $\phi_{h-s}$ convex function. If $g:\mathbb{R}\rightarrow I$ is a linear map then the composition $f\circ g:\mathbb{R}\rightarrow\mathbb{R}^\alpha$ is also generalised $\phi_{h-s}$ convex.

\item Let $f: I\rightarrow\mathbb{R}^\alpha$ be a generalised $\phi_{h-s}$ convex function. If $f$ is increasing and $g:\mathbb{R}\rightarrow I$ is convex on $\mathbb{R}$, then the composition $f\circ g:\mathbb{R}\rightarrow\mathbb{R}^\alpha$ is also generalised $\phi_{h-s}$ convex.

\end{enumerate}
\end{rmk}

We state yet another property in  the following Proposition whose proof is similar in spirit to that of Theorem 3.7 in \cite{peternfaa} and Proposition 1 (and Corollary 1) in \cite{peterfacta}.

\begin{prop}
Let $f$ be a generalised $\phi_{h_1-s}$ convex  function and $g$ be a generalised $\phi_{h_2-s}$ convex function. Let $h(t):= \max\{h_1(t), h_2(t)\}$ with $h(t)+ h(1-t)\leq c$ for a fixed positive constant $c$. If $f$ and $g$ are similarly ordered, then $fg$ is a generalised $\phi_{ch-s}$ convex function.

\end{prop}
\begin{proof}
Omitted.
\end{proof}

\begin{prop}
Let $f_n: I \rightarrow \mathbb{R}^\alpha$, $n\in\mathbb{N}$, be a  sequence of generalised $\phi_{h-s}$ convex functions converging pointwise to a function $f:I\rightarrow \mathbb{R}^\alpha$. Then $f$ is generalised convex on $I$.

\end{prop}
\begin{proof}
Let $x,y\in I$ and $0\leq t\leq 1$. It suffices to see that the generalised  $\phi_{h-s}$ convexity of $f_n$ implies
\begin{align}
f(t\phi(x)+(1-t)\phi(y)) \leq \lim_{n\rightarrow \infty} \left( \rho_{\alpha s}(t) f_{n}(\phi(x) +\rho_{\alpha s}f_n(\phi(y))   \right)  \nonumber
 = \rho_{\alpha s}(t) f(\phi(x)) +\rho_{\alpha s}(1-t) f(\phi(y))  \nonumber.
\end{align}
\end{proof}

\begin{theo}\label{t1}
Suppose $s\in(0,1)$. Define $K_\alpha(\lambda, s):=\rho_{\alpha s}(\lambda^{\frac{1}{s}})+\rho_{\alpha s}((1-\lambda)^{\frac{1}{s}})$ where $0\leq \lambda\leq 1$.

If $f$ is generalised $\phi_{h-s}$ convex, then for $0<r_1\leq r_2<\infty$, we have $f(r_1) \leq K_\alpha(\lambda, s) f(r_2)$. Moreover, if $K_\alpha(\lambda, s)\leq1^\alpha$, then $f$ is non decreasing on $(0,\infty)$.
\end{theo}

\begin{proof}
Let $x\in\mathbb{R}$ such that $\phi(x)>0$, then
\[
f(\lambda^{\frac{1}{s}}\phi(x) + (1-\lambda)^{\frac{1}{s}}\phi(x)) \leq K_{\alpha}(\lambda, s) f(\phi(x))
\]
Clearly, the function $\ell:[0,1]\rightarrow \mathbb{R}$ given by $\ell(\lambda)=\lambda^{\frac{1}{s}}+ (1-\lambda)^{\frac{1}{s}}$ is continuous on $[0,1]$, decreasing on $[0,\frac{1}{2}]$ and increasing on $[\frac{1}{2}, 1]$. For all $p_1\in [0,\frac{1}{2}]$, there exists $p_2\in[\frac{1}{2},1]$ such that $\ell(p_1)= \ell(p_2)=p_3$ where $2^{1-\frac{1}{s}}\leq p_3\leq 1$ and $\ell(\frac{1}{2}) \leq \ell(p)\leq \ell(1)$ for all $p\in[0,1]$. Consequently, we have
\begin{align}
f(t\phi(x)) \leq K_\alpha(\lambda, s) f(\phi(x)) \quad\quad\textnormal{for all} \quad t\in[2^{1-\frac{1}{s}}, 1] \quad\textnormal{and}\quad \phi(x)>0. \label{a2q3}
\end{align}
If $t\in [2^{2\left(1-\frac{1}{s}\right)}, 1]$ then $t^{\frac{1}{2}}$ lies in the interval $[2^{1-\frac{1}{s}}, 1]$ with the estimate
 \[
f(t\phi(x))= f\left(t^{\frac{1}{2}}(t^{\frac{1}{2}}\phi(x))\right) \leq K_\alpha(\lambda, s) f(t^{\frac{1}{2}}\phi(x)) \leq K^2_\alpha(\lambda, s)f(\phi(x)) \quad\quad\textnormal{for all}\quad \phi(x)>0
\]
where we have used estimate \eqref{a2q3}.
By induction, if $t\in [2^{n\left(1-\frac{1}{s}\right)}, 1]$ then $t^{\frac{1}{n}}\in [2^{1-\frac{1}{s}}, 1]$ together with the  estimate
\begin{align}
f(t\phi(x))\leq K_\alpha^n(\lambda, s) f(\phi(x)) \quad\quad\textnormal{for all}\quad \phi(x)>0\quad\textnormal {and}\quad 0< t\leq 1. \label{ap23}
\end{align}
Choose $y\in\mathbb{R}$ such that $0<\phi(x)\leq \phi(y)$  and apply the estimate \eqref{ap23}  to have
\[
f(\phi(x))=f\left(\frac{\phi(x)}{\phi(y)} \phi(y) \right) \leq K^n_\alpha(\lambda, s) f(\phi(y)).
\]
By definition, $K_\alpha(\lambda, s)>0$. Clearly, $f$ is non decreasing on $(0,\infty)$ if $K_\alpha(\lambda, s)\leq1^\alpha$.
\end{proof}

\begin{exa}\label{e1}
Let $0<s<1$ and $\beta, \gamma, \sigma \in \mathbb{R}$. Suppose that $\rho_{\alpha s}(t)+\rho_{\alpha s}(1-t)=1^\alpha$ with $t^{\alpha s} \leq \rho_{\alpha s}(t)$ for all $t\in[0,1]$. Define $f:I\subseteq [0,\infty)\rightarrow \mathbb{R}^\alpha$ by
\[
f(A)= \left\{
\begin{array}{ll}
      \beta^\alpha & \textnormal{if}\quad A=0 \\
      (\gamma A^s +\sigma)^\alpha & \textnormal{if} \quad A>0 \\
\end{array} 
\right.
\]
Then the following holds.
\begin{enumerate}[(i)]
\item If $\gamma \geq 0$ and $\sigma \leq \beta$ then $f$ is generalised $\phi_{h-s}$ convex.
\item If $\gamma \geq 0$ and $\sigma < \beta$, then $f$ is non decreasing on $(0,\infty)$.
\end{enumerate}
\end{exa}
There are only two non trivial cases of Example \ref{e1}.

Case 1: Let $A>0$ and $B>0$, then $tA+(1-t)B >0$ and
\begin{align}
f(tA+(1-t)B ) &= \gamma^\alpha(tA +(1-t) B)^{\alpha s} +\sigma^\alpha   \nonumber
\\&\leq \gamma^\alpha(t^s A^s +(1-t)^s B^s)^{\alpha } +\sigma^\alpha   \nonumber
\\&\leq  \gamma^\alpha( \rho_{\alpha s}(t) A^{\alpha s} +\rho_{\alpha s}(1-t)  B^{\alpha s}) + \sigma^\alpha(\rho_{\alpha s}(t)+\rho_{\alpha s}(1-t))  \nonumber
\\&= \rho_{\alpha s}(t)f(A) + \rho_{\alpha s}(1-t) f(B).                  \nonumber
\end{align}



Case 2: Without loss of generality, let $B>A=0$. Since $t\in[0,1]$, we have
\begin{align} 
f(t0 +(1-t) B)&= f((1-t)B) = \gamma^\alpha ((1-t)B)^{\alpha s} +\sigma^\alpha   \nonumber
\\&\leq \gamma^\alpha \rho_{\alpha s}(1-t)B^{\alpha s} +\sigma^\alpha (\rho_{\alpha s}(t)+\rho_{\alpha s}(1-t))  \nonumber
 \nonumber
\\&\leq \beta^\alpha \rho_{\alpha s}(t) + \rho_{\alpha s}(1-t)(\sigma^\alpha +\gamma^\alpha B^{\alpha s} )
 \nonumber
\\& = \rho_{\alpha s}(t)  f(0) + \rho_{\alpha s}(1-t) f(B). \nonumber 
\end{align}
This proves the first statement. The proof of the second statement is obvious.

\section{Hermite-Hadamard-type Inequalities on Fractal Sets}
We now establish some inequalities for generalised $\phi_{h-s}$ convex functions on fractal sets. 
\begin{theo}\label{p1} \text{(Generalised Hermite-Hadamard inequality).}
Let $f:I\subseteq\mathbb{R}\rightarrow\mathbb{R}^\alpha$. Suppose that $f\in C_\alpha(I)$ is a generalised $\phi_{h-s}$ convex function. Then for all $t\in [0, 1]$ and $\phi(a), \phi(b)\in I\subseteq \mathbb{R}$, the  inequality holds
\begin{align}
\frac{1}{2^\alpha\Gamma(1+\alpha)\rho_{\alpha s}\left(\frac{1}{2}\right)}f\left( \frac{\phi(a)+\phi(b)}{2}  \right) \leq \frac{ {}_{\phi(a)}{\mathcal{J}}_{\phi(b)}^{\alpha} f(\phi(x))}{(\phi(b)-\phi(a))^\alpha} \leq \left(f(\phi(a)) +f(\phi(b))\right) {}_0\mathcal{J}_1^\alpha\rho_{\alpha s}(t)    .  \label{result1}
\end{align}
\end{theo}

\begin{proof}
Using the assumption of generalised $\phi_{h-s}$ convexity on $f$, one observes that for all $t\in [0, 1]$ and $\phi(a), \phi(b)\in I\subseteq \mathbb{R}$
\begin{align}
f\left(\frac{\phi(a)+\phi(b)}{2}\right)&= f\left(\frac{1}{2}\left( t\phi(a)+ (1-t)\phi(b) + (1-t)\phi(a) +t \phi(b)   \right)\right)   \nonumber
\\&\leq    \rho_{\alpha s}(1/2) f(t\phi(a)+(1-t)\phi(b)) + \rho_{\alpha s}(1/2) f(t\phi(b)+(1-t)\phi(a))  .           \label{fracq}
\end{align}
Choosing $x\in I$ such that $\phi(x)=t\phi(a)+(1-t)\phi(b)$ and noting that $t\phi(b)+(1-t)\phi(a)= \phi(a)+\phi(b)-\phi(x)$,
 a local fractional integration of \eqref{fracq} with respect to $t$ on $[0,1]$ yields the first inequality in \eqref{result1}.
To prove the second inequality, we use the generalised $\phi_{h-s}$ convexity of $f$ to estimate the functional $\mathcal{J}^\alpha f(\phi(x))$ on $(\phi(a), \phi(b))\subseteq I$. Noting that ${}_{0}\mathcal{J}^\alpha_{1} (\rho_{\alpha s}(t)- \rho_{\alpha s}(1-t))=0$, we have
\begin{align}
{}_{\phi(a)}\mathcal{J}^\alpha_{\phi(b)} f(\phi(x)) &= (\phi(b)-\phi(a))^\alpha\,\, {}_{0}\mathcal{J}^\alpha_{1} f(t\phi(a)+(1-t)\phi(b))  \nonumber
\\& \leq (\phi(b)-\phi(a))^\alpha\,\, {}_{0}\mathcal{J}^\alpha_{1}\left[ \rho_{\alpha s}(t) f(\phi(a)) +\rho_{\alpha s}(1-t) f(\phi(b)) \right] \nonumber
\\&= (\phi(b)-\phi(a))^\alpha [f(\phi(a))+f(\phi(b))] \,\,{}_{0}\mathcal{J}^\alpha_{1}\rho_{\alpha s}(t) . \nonumber
\end{align}
This completes the proof.
\end{proof}

Per Remark \ref{r1}, several estimates can be deduced from Theorem \ref{p1}. For instance, by choosing the functions $\phi(x)=x$ for all $x\in I$ and $h(t)=1$ for all $t\in[0,1]$,  the following estimate involving Breckner $s$-convex functions in the second sense holds.
\begin{cor}\cite{mo}
Let $f:I\rightarrow\mathbb{R}^\alpha$.  Suppose $f\in C_\alpha(I)$ is a generalised $s$-convex function in the second sense for $s\in(0,1)$. Then, for all $t\in [0, 1]$ and $a, b\in I\subseteq \mathbb{R}$ with $a\neq b$, the following inequality holds
\[
2^{(s-1)\alpha} f\left( \frac{a+b}{2} \right) \leq \frac{\Gamma(1+\alpha)}{(b-a)^\alpha} {}_a \mathcal{J}_b^\alpha f(x).
\] 
\end{cor}
\begin{cor}
Let $f:I\rightarrow\mathbb{R}^\alpha$.  Suppose that $f\in C_\alpha(I)$ is a generalised  convex function. Then, for all $t\in [0, 1]$ and $a, b\in I\subseteq \mathbb{R}$ with $a\neq b$, the following inequality holds
\[
\frac{1}{\Gamma(1+\alpha)} f\left( \frac{a+b}{2} \right) \leq \frac{   {}_a \mathcal{J}_b^\alpha f(x)}{(b-a)^\alpha}\leq \frac{\Gamma(1+\alpha)}{\Gamma(1+2\alpha)} (f(a) +f(b)) .
\] 
\end{cor}

\begin{cor}
Let $f:I\rightarrow\mathbb{R}^\alpha$. Suppose that $f\in C_\alpha(I)$ is a generalised $P$-function. Then, for all $t\in [0, 1]$ and $a, b\in I\subseteq \mathbb{R}$ with $a\neq b$, the following inequality holds
\[
\frac{2^{-\alpha}}{\Gamma(1+\alpha)} f\left( \frac{a+b}{2} \right) \leq \frac{   {}_a \mathcal{J}_b^\alpha f(x)}{(b-a)^\alpha}\leq \frac{1}{\Gamma(1+\alpha)} (f(a) +f(b)) .
\] 
\end{cor}

\subsection{Extensions of the Hermite-Hadamard Inequalities on Fractal Sets}
We now establish an important representation lemma for  functions belonging to the space $ C_\alpha(I)$.
\begin{lem} \label{ojoa}

Let $\phi:[a,b]\rightarrow (0,\infty)$ and $f:I\subseteq\mathbb{R}\rightarrow\mathbb{R}^\alpha$ be such that $f\in C_\alpha(I)$.
Then for $0\leq \lambda \leq 1$, we have
\begin{align}
{}_{0}\mathcal{J}_1 ^{\alpha} f\left((1-t)\phi(a) +t\phi(b)  \right) =(1-\lambda)^\alpha {}_{0}\mathcal{J}_{1} ^\alpha f\left((1-t)((1-\lambda)\phi(a) +\lambda\phi(b))+t\phi(b)\right)  \nonumber
\\\quad +\lambda^\alpha {}_{0}\mathcal{J}_1^\alpha  f\left((1-t)\phi(a) +t((1-\lambda)\phi(a) +\lambda\phi(b)) \right)   \label{ineq1}
\end{align}
\end{lem}

\begin{proof}
It is easy to see that equality \eqref{ineq1} holds for $\lambda=0$ and $\lambda =1$. Now, suppose $\lambda\in(0,1)$. The change of variable 
 $u= \lambda(1-t) +t$ yields 
\begin{align}
{}_{0}\mathcal{J}_{1} ^\alpha\,\, f\left((1-t)((1-\lambda)\phi(a) +\lambda\phi(b))+t\phi(b)\right) &=\frac{1}{\Gamma(\alpha+1)}\frac{1}{(1-\lambda)^\alpha} \int_\lambda^1 f\left( (1-u) \phi(a) +u\phi(b) \right)\,(du)^\alpha  \nonumber
\\&= \frac{1}{(1-\lambda)^\alpha}\,\, {}_0 \mathcal{J}_\lambda^\alpha\,\, f\left( (1-u) \phi(a) +u\phi(b) \right). \label{one1}
\end{align}
On the other hand, the change of variable $u=\lambda t$ yields
\begin{align}
{}_{0}\mathcal{J}_1^\alpha \,\, f\left((1-t)\phi(a) +t((1-\lambda)\phi(a) +\lambda\phi(b)) \right) &= \frac{1}{\Gamma(1+\alpha)}\frac{1}{\lambda^\alpha} \int_0^\lambda f((1-u)\phi(a) + u\phi(b) )\,\, (du)^\alpha \nonumber
\\&= \frac{1}{\lambda^\alpha}\,\, {}_0\mathcal{J}_\lambda^\alpha\,\, f((1-u)\phi(a) +u\phi(b)). \label{two2}
\end{align}
Combining \eqref{one1} and \eqref{two2} gives the identity \eqref{ineq1}. This completes the proof.

\end{proof}

\begin{theo}
Let $f:I\subseteq\mathbb{R}\rightarrow\mathbb{R}^\alpha$. Suppose that $f\in C_\alpha(I)$ is a generalised $\phi_{h-s}$ convex function. If $\rho_{\alpha s}(t)\leq t^\alpha $  for all $t\in [0, 1]$ then for all $\phi(a), \phi(b)\in I\subseteq \mathbb{R}$ with $\phi(a)\leq \phi(b)$ and any $\lambda \in[0,1]$, the following inequality holds
\begin{align}
\frac{f\left( \frac{\phi(a)+\phi(b)}{2} \right)}{\Gamma(\alpha+1)}  
&\leq \frac{\rho_{\alpha s}(1-\lambda)}{\Gamma(\alpha+1)} f\left( \frac{(1-\lambda)\phi(a) +(\lambda +1)\phi(b)}{2} \right) + \frac{\rho_{\alpha s}(\lambda )}{\Gamma(\alpha +1)}f\left( \frac{(2-\lambda)\phi(a)+\lambda \phi(b)}{2} \right) \nonumber
\\&\leq \frac{ {}_{\phi(a)}\mathcal{J}_{\phi(b)}^\alpha f(x) }{(\phi(b)-\phi(a))^\alpha}  \nonumber
\\&\leq \left[  f((1-\lambda)\phi(a)+\lambda^\alpha f(\phi(a)) +\lambda \phi(b)) +(1-\lambda)^\alpha f(\phi(b))   \right]\,\,{}_0\mathcal{J}_{1}^\alpha\, \rho_{\alpha s}(x)   \nonumber
\\& \leq [f(\phi(a)) +f(\phi(b))]\,\, {}_0\mathcal{J}_{1}^\alpha\, \rho_{\alpha s}(x)  . \nonumber
\end{align}
\end{theo}

\begin{proof}
Estimate \eqref{result1} and the generalised $\phi_{h-s}$ convexity of $f$ imply
\begin{align}
&\frac{1}{\Gamma(\alpha+1)}f\left( \frac{(1-\lambda)\phi(a) + (1+\lambda)\phi(b)}{2} \right) = \frac{1}{\Gamma(\alpha+1)}f\left( \frac{(1-\lambda)\phi(a) + \lambda\phi(b)}{2} +\frac{\phi(b)}{2} \right)  \nonumber
\\& \leq \frac{2^\alpha\rho_{\alpha s}(1/2)}{\Gamma(\alpha +1)}\int_0^1 f\Big((1-t)[(1-\lambda)\phi(a) +\lambda\phi(b)] + t \phi(b)  \Big) (dt)^\alpha  \nonumber
\\& \leq  \Big( f[(1-\lambda)\phi(a) +\lambda \phi(b)] + f(\phi(b))  \Big) \, {}_0 \mathcal{J}_1^\alpha  \rho_{\alpha s} (t) \label{qw1}
\end{align}
where we have used the fact that $0^\alpha\leq 2^\alpha\rho_{\alpha s}(1/2)\leq 1^\alpha$. Using again estimate \eqref{result1} and the generalised $\phi_{h-s}$ convexity of $f$ we have
\begin{align}
& \frac{1}{\Gamma(\alpha+1)}f\left(\frac{(2-\lambda)\phi(a) +\lambda \phi(b)}{2}  \right) =  \frac{1}{\Gamma(\alpha+1)}f\left(\frac{\lambda\phi(a)}{2} +\frac{(1-\lambda)\phi(a)+\lambda \phi(b)}{2}  \right)  \nonumber
\\& \leq \frac{2^\alpha\rho_{\alpha s}(1/2)}{\Gamma(\alpha +1)}\int_0^1 f\Big((1-t)\phi(a) +t[(1-\lambda)\phi(a)+ \lambda \phi(b)  ]\Big) (dt)^\alpha  \nonumber
\\& \leq \Big(f(\phi(a))+f[(1-\lambda)\phi(a)+\lambda\phi(b)]\Big) \,\,{}_0\mathcal{J}_1^\alpha \rho_{\alpha s} (t).\label{qw2}
\end{align}
Multiply the estimates \eqref{qw1} by $(1-\lambda)^\alpha$ and \eqref{qw2} by $\lambda^\alpha$, and add the resulting double estimates.  Calling upon Lemma \ref{ojoa}, we find
\begin{align}
& \frac{(1-\lambda)^\alpha}{\Gamma(\alpha +1)} f\left( \frac{(1-\lambda)\phi(a) +(\lambda+1)\phi(b)}{2} \right) +\frac{\lambda^\alpha}{\Gamma(\alpha+1)} f\left( \frac{(2-\lambda)\phi(a)+\lambda\phi(b)}{2} \right)  \nonumber
\\& \leq 2^\alpha\rho_{\alpha s}(1/2)\,\, {}_0 \mathcal{J}_1^\alpha\, f\big((1-t)\phi(a) +t\phi(a)  \big)  \nonumber
\\&\leq \Big(f[(1-\lambda)\phi(a)+\lambda\phi(b)]  +(1-\lambda)^\alpha f(\phi(b)) +\lambda^\alpha f(\phi(a)) \Big) {}_0\mathcal{J}_1\, \rho_{\alpha s}(t).  \label{c1}
\end{align}

We use the generalised $\phi_{h-s}$ convexity of $f$ to obtain the estimates
\begin{align}
&(1-\lambda)^\alpha f\left( \frac{(1-\lambda)\phi(a)+(\lambda+1)\phi(b)}{2} \right) +\lambda^\alpha f\left( \frac{(2-\lambda)\phi(a) +\lambda\phi(b)}{2}  \right)  \nonumber
\\& \geq \rho_{\alpha s}(1-\lambda) f\left( \frac{(1-\lambda)\phi(a)+(\lambda+1)\phi(b)}{2} \right) +\rho_{\alpha s}(\lambda) f\left( \frac{(2-\lambda)\phi(a) +\lambda\phi(b)}{2}  \right)  \nonumber
\\& \geq f\left( (1-\lambda) \left( \frac{(1-\lambda)\phi(a) +(\lambda+1)\phi(b)}{2} \right) +\lambda\left(  \frac{(2-\lambda)\phi(a) +\lambda\phi(b)}{2} \right) \right)  \nonumber
\\&=   f\left( \frac{\phi(a)+\phi(b)}{2} \right). \label{c2}
\end{align}

 Accordingly, the condition $\rho_{\alpha}(x) \leq x^\alpha$ for all $x,s \in[0,1]$ yields
\begin{align}
&\Big( f[(1-\lambda)\phi(a)+\lambda\phi(b)] +(1-\lambda)^\alpha f(\phi(b)) +\lambda^\alpha f(\phi(a))  \Big)\,{}_0\mathcal{J}^\alpha_1 \rho_{\alpha s}(t)  \nonumber
\\&\leq \Big(\rho_{\alpha s}(1-\lambda) f(\phi(a)) +\rho_{\alpha s}(\lambda) f(\phi(b)) +(1-\lambda)^\alpha f(\phi(b)) +\lambda^\alpha f(\phi(a))   \Big) \,{}_0\mathcal{J}^\alpha_1 \rho_{\alpha s}(t)  \nonumber
\\&= \Big( f(\phi(a)) +f(\phi(b)) \Big) \,{}_0\mathcal{J}^\alpha_1 \rho_{\alpha s}(t) . \label{c3}
\end{align}
To complete the proof, we combine estimates \eqref{c1}, \eqref{c2} and \eqref{c3} and note that a suitable change of variable yield ${}_0 \mathcal{J}_1^\alpha\, f\big((1-t)\phi(a) +t\phi(a)  \big)=  \frac{ {}_{\phi(a)}\mathcal{J}_{\phi(b)}^\alpha f(x) }{(\phi(b)-\phi(a))^\alpha} $.

\end{proof}

Next, we obtain average refinements of Hermite-Hadamard inequalities by using the generalised notion of convexity  and assuming that $\rho_{s}:[0,1]\rightarrow (0,\infty)$ is linear and multiplicative (c.f. 
\cite{fech, anji, hoff, roit}. Note that a multiplicative function $\rho$ satisfies $\rho(1)=1$. An example of a multiplicative, linear function is $\rho_s(x)=x$).

\begin{theo}\label{k9}
Let $f:I\rightarrow\mathbb{R}^\alpha$. Suppose that $f\in C_\alpha(I)$ is a generalised $\phi_{h-s}$ convex function and $\rho_{s}:[0,1]\rightarrow (0,\infty)$ is linear and multiplicative. Then for all $t\in[0,1]$, the following inequalities hold
\begin{align}
X_{[1-t,t,C]}\leq f((1-t)\phi(a)+t\phi(b))\leq X_{[1-t,t,c]}  \label{es2}
\end{align}
where c=\textnormal{ min}\{t,\,1-t\} , C=\textnormal{max}\{t,\,1-t\} and 
\[
X_{[u,v,\varrho]}:= \rho_{\alpha s}(u) f(\phi(a)) +\rho_{\alpha s} (v) f(\phi(b)) -\rho_{\alpha s}(2\varrho)\left(\frac{f(\phi(a))+f(\phi(b))}{2} -f\left( \frac{\phi(a)+\phi(b)}{2} \right)  \right)
\] with $0\leq u,v,\varrho\leq 1$. Furthermore, the double inequality hold
\begin{align}
X_{[1,1,2C]} -f(\phi(x)) \leq f(\phi(a)+\phi(b)-\phi(x)) \leq X_{[1,1,2c]} -f(\phi(x)) \label{ajz}
\end{align}
where, in particular, $c=\textnormal{min}\left\{\frac{\phi(b)-\phi(x)}{\phi(b)-\phi(a)}, \frac{\phi(x)-\phi(a)}{\phi(b)-\phi(a)}   \right\}$  and $C=\textnormal{max}\left\{\frac{\phi(b)-\phi(x)}{\phi(b)-\phi(a)}, \frac{\phi(x)-\phi(a)}{\phi(b)-\phi(a)}   \right\}$.  

Moreover, the refined Hermite-Hadamard inequality also holds
\begin{align}
&\frac{1}{\Gamma(\alpha +1)}\int_0^1 X_{[1,\, 1,\, 2C_1+2C_2]}(dt)^\alpha \,\,  \nonumber
\\&\leq f\left( \phi(a) +\phi(b) -\frac{1}{2}(\phi(u)+\phi(v) )  \right)+\frac{{}_{\phi(v)}\mathcal{J}_{\phi(u)}f }{(\phi(v)-\phi(u))^\alpha}\,\, \nonumber
\\&\leq \frac{1}{\Gamma(\alpha +1)}\int_0^1 X_{[1,\, 1,\, 2c_1+2c_2]}(dt)^\alpha \,\,  \nonumber
\end{align}
where $C_1:=C_1(t)=(1-t)\textnormal{max}\left\{t, 1-t   \right\}$,  $C_2:=C_2(t)=t\,\textnormal{max}\left\{t, 1-t   \right\} $,  $c_1:=c_1(t)=(1-t)\textnormal{min}\left\{t, 1-t   \right\}$ and $c_2:=c_2(t)=t\,\textnormal{min}\left\{t, 1-t   \right\} $.
\end{theo}

\begin{proof}
Let $t\in[0,\frac{1}{2}] $, then $c=t$. The linearity of $\rho_{ s}$ guarantees that
\begin{align}
X_{[1-t,t,c]} &=   \rho_{\alpha s}(1-2t)f(\phi(a)) +\rho_{\alpha s}(2t) f\left( \frac{\phi(a)+\phi(b)}{2}  \right)     \nonumber
\\&\geq f\left( (1-2t)\phi(a) +2t\left( \frac{\phi(a)+\phi(b)}{2} \right)  \right)       \nonumber
\\& = f((1-t)\phi(a)+ t\phi(b)).  \label{u2}
\end{align}
The case $t\in[\frac{1}{2}, 1]$ is proved by replacing $t$ by $1-t$ in the argument above. Next, we use the generalised $\phi_{h-s}$ convexity of $f$  to estimate 
\begin{align}
f(\phi(a))&= f\left[\frac{1}{t+1}((t+1)\phi(a)-t\phi(b) )+\frac{t}{t+1}\phi(b)   \right]  \nonumber
\\& \leq \rho_{\alpha s}\left(\frac{1}{t+1} \right) f((t+1)\phi(a) -t\phi(b)) +\rho_{\alpha s}\left( \frac{t}{t+1} \right) f(\phi(b)) .\nonumber
\end{align}
The latter together with the multiplicative property of $\rho_{ s}$   imply the estimate
\begin{align}
\rho_{\alpha s}(t+1) f(\phi(a)) -\rho_{\alpha s} (t) f(\phi(b)) \leq f((t+1) \phi(a) -t\phi(b)).\label{u1}
\end{align}
On the other hand, if $t\in[0,\frac{1}{2}]$ (the case $t\in[\frac{1}{2},1]$ follows by a similar argument) then $C= 1-t$. Calling upon estimate \eqref{u1}, we have
\begin{align}
 &f((1-t)\phi(a)+t\phi(b))=f\left((2-2t)\frac{\phi(a)+\phi(b)}{2} -(1-2t) f(\phi(b))   \right)  \nonumber
\\&\geq\rho_{\alpha s}(1+(1-2t)) f\left( \frac{\phi(a)+\phi(b)}{2} \right) -\rho_{\alpha s} (1-2t) f(\phi(b)) \nonumber
\\&=\rho_{\alpha s}(1-t)f(\phi(a)) +\rho_{\alpha s }(t) f(\phi(b)) -\rho_{\alpha s}(2C) \left(\frac{f(\phi(a))+ f(\phi(b))}{2}- f\left( \frac{\phi(a)+\phi(b)}{2}  \right)    \right) . \label{u3}
\end{align}
Combining estimates \eqref{u3} and \eqref{u2} proves the double inequality \eqref{es2}.

Choose $x\in[a,b]$ such that $\phi(a)\leq \phi(x)\leq \phi(b)$ and $\phi(x)=t\phi(a)+(1-t)\phi(b)$ where $t=\frac{\phi(b)-\phi(x)}{\phi(b)-\phi(a)}$. Note that $c=\textnormal{min}\left\{\frac{\phi(b)-\phi(x)}{\phi(b)-\phi(a)}, \frac{\phi(x)-\phi(a)}{\phi(b)-\phi(a)}   \right\}$. Applying estimate \eqref{u2} twice, we have
\begin{align}
f(\phi(a)+\phi(b)-\phi(x))&= f((1-t)\phi(a)+t\phi(b))              \nonumber
\\&\leq X_{[1-t,t,c]}= X_{[1-t,1-(1-t),c]}   \nonumber
\\& =X_{[1,1,c]} -\rho_{\alpha s}(t) f(\phi(a)) -\rho_{\alpha s} (1-t) f(\phi(b))   \nonumber
\\&\leq X_{[1,1,2c]} -f(\phi(x)). \label{ana}
\end{align}
Note that by using \eqref{u3} instead of \eqref{u2}, the following reverse estimate of \eqref{ana} holds

\[
X_{[1,1,2C]} -f(\phi(x)) \leq f(\phi(a)+\phi(b)-\phi(x)).
\]
Combining the latter with \eqref{ana} proves estimate \eqref{ajz}.

We extend \eqref{ana} to estimates with two parameters. The generalised $\phi_{h-s}$ convexity of $f$ implies
\begin{align}
f\left(  \phi(a) +\phi(b) -((1-t)\phi(x)+t\phi(y) )   \right) &= f\left((1-t)(\phi(a)+\phi(b) -\phi(x) )   +t(\phi(a)+\phi(b)-\phi(y) )   \right)   \nonumber
\\&\leq \rho_{\alpha s}(1-t) f(\phi(a)+\phi(b) -\phi(x))   +  \rho_{\alpha s}(t) f(\phi(a)+\phi(b) -\phi(y)) .  \nonumber
\end{align}
The latter together with \eqref{ana} yield
\begin{align}
f\left(  \phi(a) +\phi(b) -((1-t)\phi(x)+t\phi(y) )   \right) & \leq \rho_{\alpha s}(1-t) \left( X_{[1,1,2c]}-f(\phi(x))\right)   +  \rho_{\alpha s}(t) \left( X_{[1,1,2c]}-f(\phi(y))\right) \nonumber
\\&= X_{[1-t,t,2c_1]} +X_{[t,1-t,2c_2]} -  \rho_{\alpha s}(1-t) f(\phi(x))     -  \rho_{\alpha s}(t) f(\phi(y))      \label{y6}
\end{align}
where $c_1:=c_1(t)=(1-t)\textnormal{min}\left\{t, 1-t   \right\}$ and $c_2:=c_2(t)=t\,\textnormal{min}\left\{t, 1-t   \right\} $.
\vskip3mm
\noindent
The generalised $\phi_{h-s}$ convexity of $f$ implies
\begin{align}
&f\left( \phi(a) +\phi(b) -\frac{1}{2}(\phi(x)+\phi(y) )  \right)  \nonumber
\\&= f\left(\frac{1}{2}(\phi(a)+\phi(b) -((1-t)\phi(x) +t\phi(y) ) ) +\frac{1}{2}(\phi(a)+\phi(b) -(t\phi(x) +(1-t)\phi(y) ) )    \right)  \nonumber
\\&\leq \rho_{\alpha s}\left( \frac{1}{2} \right)f\left( \phi(a)+\phi(b) -((1-t)\phi(x) +t\phi(y) )  \right) +\rho_{\alpha s}\left( \frac{1}{2} \right)f\left( \phi(a)+\phi(b) -(t\phi(x) +(1-t)\phi(y) )  \right)  \nonumber
\end{align}

Combining the latter  with \eqref{y6} and using the multiplicative property of $\rho_s$, we have
\begin{align}
&f\left( \phi(a) +\phi(b) -\frac{1}{2}(\phi(x)+\phi(y) )  \right)  \nonumber
\\&\leq   X_{[1-t,t,2c_1]} +X_{[t,1-t,2c_2]}-\rho_{\alpha s} \left( \frac{1}{2}\right)(f(\phi(x))+ f(\phi(y)))  \nonumber
\\& =f(\phi(a)) +f(\phi(b)) -\frac{1}{2}[f(\phi(x))+f(\phi(y))] -4\rho_{\alpha s}(c_1+c_2)\left(\frac{f(\phi(a))+f(\phi(b))}{2} -f\left( \frac{\phi(a)+\phi(b)}{2} \right)  \right)
\label{latter}
\end{align}

Similarly, the following reverse estimate holds
 \begin{align}
&f\left( \phi(a) +\phi(b) -\frac{1}{2}(\phi(x)+\phi(y) )  \right)  \nonumber
\\&\geq f(\phi(a)) +f(\phi(b)) -\frac{1}{2}[f(\phi(x))+f(\phi(y))] -4\rho_{\alpha s}(C_1+C_2)\left(\frac{f(\phi(a))+f(\phi(b))}{2} -f\left( \frac{\phi(a)+\phi(b)}{2} \right)  \right)  \label{latest}
\end{align}
where $C_1:=C_1(t)=(1-t)\textnormal{max}\left\{t, 1-t   \right\}$ and $C_2:=C_2(t)=t\,\textnormal{max}\left\{t, 1-t   \right\} $.
Combining estimates \eqref{latter} and \eqref{latest} and applying a local fractional integration  w.r.t $t$ on $[0,1]$ yield

\begin{align}
&\frac{1}{\Gamma(\alpha +1)}\int_0^1 X_{[1,\, 1,\, 2C_1+2C_2]}(dt)^\alpha -\frac{{}_{\phi(v)}\mathcal{J}_{\phi(u)}f}{(\phi(v)-\phi(u))^\alpha}\,\,\,\,  \nonumber
\\&\leq f\left( \phi(a) +\phi(b) -\frac{1}{2}(\phi(u)+\phi(v) )  \right)  \nonumber
\\&\leq \frac{1}{\Gamma(\alpha +1)}\int_0^1 X_{[1,\, 1,\, 2c_1+2c_2]}(dt)^\alpha -\frac{{}_{\phi(v)}\mathcal{J}_{\phi(u)}f}{(\phi(v)-\phi(u))^\alpha}\,\,\,\,  \nonumber
\end{align}
where we have used the change of variables $\phi(x)=t\phi(u)+(1-t)\phi(v)$ and $\phi(y)=(1-t)\phi(u)+t\phi(v)$ with $u,v\in I$. This completes the proof.
\end{proof}

\subsection{Application to Probability}
Let $\chi$ be a random variable whose  probability density function $\mathfrak{p}:[t_1, t_2] \rightarrow [0,1]^\alpha$ is generalised $\phi_{h-s}$ convex with a cummulative distribution function given by
$${\bf P}_\alpha(\chi \leq x)=\mathcal{F}_\alpha(x) :=\frac{1}{\Gamma(\alpha +1)} \int_{t_1}^x \mathfrak{p}(\theta)\, (d\theta)^\alpha$$
and the fractional normalising condition $\frac{1}{\Gamma(\alpha+1)}\int_{-\infty}^{\infty} \mathfrak{p}(\theta)\, (d\theta)^\alpha =1$. Define the  functional 
$$\mathcal{E}_{\alpha s}(\chi):= \frac{1}{\Gamma(\alpha+1)}\int_{t_1}^{t_2} \rho_{\alpha s}(\theta)\, \mathfrak{p}(\theta)\, (d\theta)^\alpha$$
where $\rho_{\alpha s}:[0,\infty)\rightarrow [0,\infty)^\alpha$ is as  in the previous sections. With $\rho_{\alpha s}(\theta)=\theta^\alpha$ and $\rho_{\alpha s}(\theta)=\theta^{\alpha s} $, the functional $\mathcal{E}_{\alpha s}({\chi})$ reduces to  generalised expectation and $s$-moments of fractional order respectively. In particular, $\mathcal{E}_{\alpha s}(\chi)$ satisfies
$$\mathcal{E}_{\alpha s}(x) =\frac{1}{\Gamma(\alpha+1)}\int_{t_1}^x \theta^\alpha\, d\mathcal{F}_\alpha (\theta)= t_2^\alpha -\frac{1}{\Gamma(\alpha+1)} \int_{t_1}^{t_2} \mathcal{F}_\alpha (\theta) \, (d\theta)^\alpha.$$
 Interested readers may see \cite{ju, abd} for more information on probability using fractional calculus.

\begin{theo}
Under the assumption of Theorem \ref{p1}, the following inequalities hold
\begin{align}
\frac{2^{\alpha(s-1)}}{\Gamma(\alpha +1)} {\bf P}_\alpha\Big( \chi \leq \frac{\phi(a)+\phi(b)}{2} \Big) &\leq \frac{(\phi(b))^\alpha-\mathcal{E}_\alpha(\chi)}{(\phi(b)-\phi(a))^\alpha \Gamma(\alpha+1)} \nonumber
\\&\leq \frac{\Gamma(\alpha s +1)}{\Gamma(1+\alpha(s+1))}\Big( {\bf P}_\alpha(\chi\leq \phi(a)) + {\bf P}_\alpha(\chi\leq \phi(b))  \Big). \nonumber
\end{align}
\end{theo}
\begin{proof}
Follows easily by using Theorem \ref{p1} and the identity $\frac{1}{\Gamma(\alpha+1)}\int_0^1 t^{\alpha s}\, (dt)^\alpha -\frac{\Gamma(\alpha s+1)}{\Gamma(\alpha(s+1)+1)}=0^\alpha.$
\end{proof}

\end{document}